\documentclass[11pt]{article}
\usepackage[utf8]{inputenc}

\usepackage{textcomp}
\usepackage[russian, english]{babel}

\usepackage{amsmath, amsfonts, amssymb,amsthm}
\usepackage{verbatim}

\usepackage{float}
\usepackage{changepage}
\usepackage{array}
\usepackage{enumerate}
\usepackage{hyperref}
\usepackage{authblk}
\usepackage{caption}

\usepackage[nottoc,numbib]{tocbibind}

\usepackage{xcolor}

\hypersetup{colorlinks,linkcolor={blue},citecolor={red},urlcolor={blue}}

\usepackage{tikz}
\usetikzlibrary{calc, matrix, arrows,decorations.pathmorphing, positioning}
\numberwithin{equation}{section}
\providecommand{\keywords}[1]
{
  \small	
  \textbf{\textit{Keywords---}} #1
}

\theoremstyle{definition}

\newtheorem{lemma}{Lemma}[section]

\newtheorem{definition}{Definition}[section]
\newtheorem{remark}{Remark}[section]

\usepackage{etoolbox}
\makeatletter
\AfterEndEnvironment{definition}{\@doendpe}
\AfterEndEnvironment{theorem}{\@doendpe}
\AfterEndEnvironment{lemma}{\@doendpe}
\AfterEndEnvironment{prop}{\@doendpe}
\AfterEndEnvironment{remark}{\@doendpe}
\AfterEndEnvironment{example}{\@doendpe}
\AfterEndEnvironment{consequence}{\@doendpe}
\makeatother

\newcommand{\beq}{\begin{equation}}
\newcommand{\ee}{\end{equation}}

\newcommand{\ben}{\begin{equation*}}
\newcommand{\een}{\end{equation*}}

\newcommand{\normord}[1]{:\mathrel{\mkern2mu #1 \mkern2mu}:}

\newcommand\myeq{\mathrel{\stackrel{\makebox[0pt]{\mbox{\normalfont\tiny def}}}{=}}}

\DeclareMathOperator*{\Tr}{Tr}
\DeclareMathOperator*{\Cl}{Cl}

\DeclareMathOperator*{\End}{End}

\usepackage{paralist}

  \pltopsep=\medskipamount
\usepackage[title]{appendix}

\newcolumntype{M}[1]{>{\centering\arraybackslash}m{#1}}

\newcommand{\qbinom}[2]{\genfrac{[}{]}{0pt}{}{#1}{#2}}

\usepackage[explicit]{titlesec}
\titleformat{\section}
  {\normalfont\normalsize\bfseries\centering}{\thesection}{1em}{\MakeUppercase{#1}}

\title{An alternative proof of $\widehat{\mathfrak{sl}}_2'$ standard module semi-infinite structure}

\author[]{Timur Kenzhaev \thanks{kenzhaev\_t\_d@mail.ru}}

\affil[]{Skolkovo Institute of Science and Technology, Moscow, Russia}

\date{}

\begin{document}

\pagenumbering{arabic}

\maketitle

\begin{abstract}
B. Feigin and A. Stoyanovsky found the basis of semi-infinite monomials in standard $\widehat{\mathfrak{sl}}_2'$-module $L_{(0, 1)}$ with Lefschetz formula for the corresponding flag variety. These semi-infinite monomials are constructed by modes of the current $e(z) = \sum\limits_{n\in\mathbb{Z}} e_n\,z^{- n - 1}$. We give an alternative proof of this fact using explicit ``fermionic" construction of this module. Namely, we realize $L_{(0, 1)}$ inside of the zero-charge subspace of Fermionic Fock space and show linear independence of vectors corresponding to semi-infinite monomials.      
\end{abstract}


\section{Introduction}

Lie algebra $\widehat{\mathfrak{sl}_2}'$ is
\ben
\widehat{\mathfrak{sl}_2}' = \mathfrak{sl}_2\otimes\mathbb{C}[t, t^{-1}]\oplus\mathbb{C}\,K
\een
with bracket
\ben
\begin{aligned}
&[e_n, e_m] = [f_n, f_m] = 0, &[e_n, f_m] &= h_{n + m} + n\,\delta_{n, -m}\,K,
\\
&[h_n, e_m] = 2\,e_{n + m}, &[h_n, f_m] &= -2\,f_{n + m}, 
\\
&[h_n, h_m] = 2n\,\delta_{n, -m}\,K, &[K, \cdot\,\,] &= 0.
\end{aligned}
\een
Triangular decomposition is $\widehat{\mathfrak{sl}_2}' = \hat{\mathfrak{n}}_+\oplus \hat{\mathfrak{h}}\oplus \hat{\mathfrak{n}}_-$, where
\ben
\hat{\mathfrak{n}}_+ = \langle e_0\rangle + \sum\limits_{k > 0}\,t^k\,\mathfrak{sl}_2, \quad \hat{\mathfrak{h}} = \langle h_0, K \rangle, \quad \hat{\mathfrak{n}}_- = \langle f_0\rangle + \sum\limits_{k > 0}\,t^{-k}\,\mathfrak{sl}_2\:.
\een
\begin{definition}
$\widehat{\mathfrak{sl}_2}'$-module of the highest weight $(l, k)$ is the irreducible module $L_{(l, k)}$ with cyclic vector $v\in L_{(l, k)}$ s.t. 
\ben
\hat{\mathfrak{n}}_+\, v = 0, \quad h_0\,v = l v, \quad K\,v = k\,v.
\een
\end{definition}
Irreducible $\widehat{\mathfrak{sl}_2}'$ module of the highest weight $(0, 1)$~---~$L_{(0, 1)}$ is called \emph{standard}. It could be realized as direct sum of bosonic Fock modules:
\ben
L_{(0, 1)} = \bigoplus\limits_{m\in\mathbb{Z}} F_{m\sqrt{2}}.
\een
In terms of the currents
\ben
e(z) = \sum\limits_{n\in\mathbb{Z}}\,e_n\,z^{- n - 1},\quad f(z) = \sum\limits_{n\in\mathbb{Z}}\,f_n\,z^{- n - 1},\quad h(z) = \sum\limits_{n\in\mathbb{Z}}\,h_n\,z^{- n - 1},
\een
action is given by 
\ben
e(z)= \;\normord{\exp\left(\sqrt{2}\,\varphi(z)\right)}\,, \quad f(z)= \;\normord{\exp\left(-\sqrt{2}\,\varphi(z)\right)}\,, \quad h(z)= \;\sqrt{2}\,\partial\varphi(z),  
\een
where $\varphi(z)$ is holomorphic bosonic field:
\ben
\varphi(z) = q + a_0\,\log z + \sum\limits_{n\neq 0} \frac{a_n}{-n}\,z^{-n},
\een
with $[a_0, q] = 1$.
Character of $L_{(0, 1)}$ is
\beq
\label{CharacterL01}
\ch L_{(0, 1)} = \Tr \left(z^{\frac{h_0}{2}}\, q^{L_{0}}\right) = \sum\limits_{n\in\mathbb{Z}}\, \frac{z^n\, q^{n^2}}{(q)_{\infty}},
\ee
where $L_{0}$ is zero mode of stress-energy tensor\
\ben
L(z) = \sum\limits_{n\in\mathbb{Z}}\,L_n\,z^{- n - 2} = \frac{\normord{a(z)^2}}{2}.
\een
\begin{definition}
\emph{Basic subspace} $W_{0} \subset L_{(0, 1)}$ is $W_0 = \mathbb{C}[e_{-1}, e_{-2}, e_{-3}, \ldots]|0\rangle$.
\end{definition}
Feigin and Stoyanovsky proved in \cite{FS} that character of this subspace equals
\beq
\label{BasicSubspaceCharacter}
\ch W_0 = \sum\limits_{n = 0}^{\infty}\,\frac{z^n\, q^{n^2}}{(q)_n}.
\ee
Acting by translation subgroup $T(x) = x + 1, \: x\in\mathbb{R}$ of  Weyl group we get sequence of embedded subspaces 
\ben
W_m = T^m\,W_0,
\een
$W_{m + 1} \subset W_{m}$ with character 
\beq
\label{BasicSubspaceCharacter_j}
\ch W_m = T^{m}\,\ch W_0 = \Tr\left(z^{\frac{h_0}{2} + m}\,q^{L_0 + n\,h_0 + n^2}\right)|_{W_0} = \sum\limits_{n = 0}^{\infty}\, \frac{z^{m + n}\,q^{m^2 + 2mn + n^2}}{(q)_{n}} = \sum\limits_{n = m}^{+\infty}\frac{z^{n}\, q^{n^2}}{(q)_{n - m}},
\ee
setting $m\to -\infty$ we get
\beq
\label{CharacterOfSum}
\ch \left(\bigoplus\limits_{m\in\mathbb{Z}}\, W_m\right) = \lim\limits_{m\to -\infty}\,\sum\limits_{n = m}^{\infty}\, \frac{z^{n}\, q^{n^2}}{(q)_{n - m}} = \sum\limits_{n\in\mathbb{Z}}\, \frac{z^n\, q^{n^2}}{(q)_{\infty}}. 
\ee
Thus, comparing with formula \eqref{CharacterL01}, we get
\ben
W = \bigoplus\limits_{j \in \mathbb{Z}}\,W_j = L_{(0, 1)}.
\een

\begin{figure}[h!]
\begin{center}
\includegraphics[scale = 0.5]{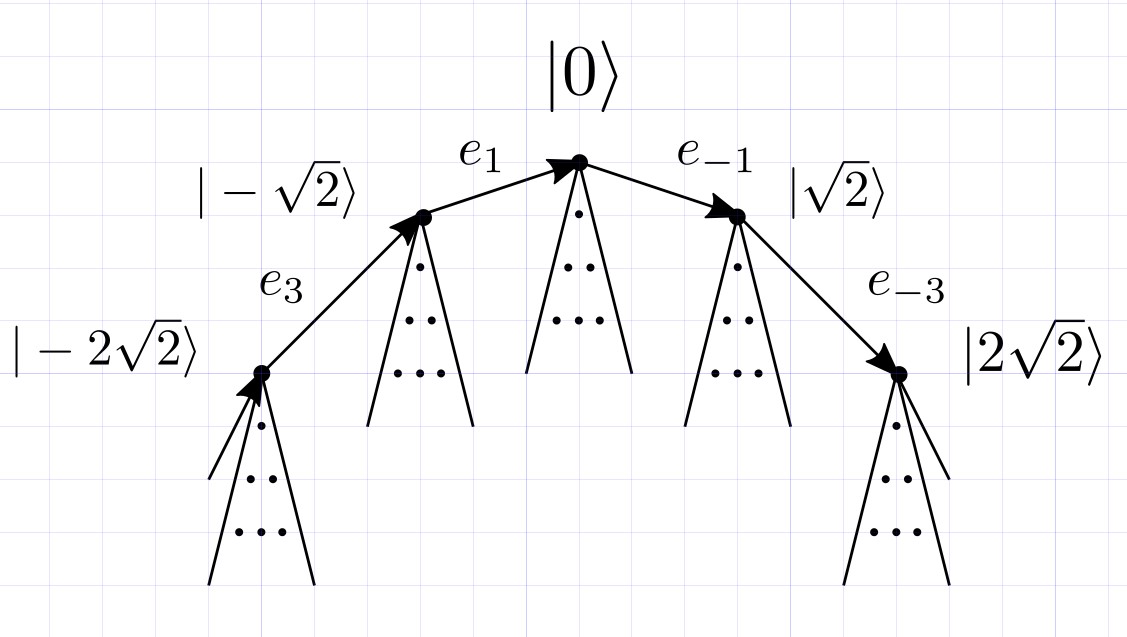}
\caption{Weight diagram for $L_{(0, 1)}$
\\
$\ch L_{(0, 1)} = \sum\limits_{m\in\mathbb{Z}}\,\frac{z^m\,q^{m^2}}{(q)_{\infty}}$
}
\end{center}
\end{figure}

\begin{definition}
\label{Fibb_monomials}
Consider polynomial ring $\mathbb{C}[x_i \:|\: i\in\mathbb{Z}]$. Monomial $x_{j_1}\,x_{j_2}\ldots x_{j_k} $ is called \emph{Fibonacci-1 monomial} if  $j_m - j_{m - 1} > 1$ for any $m\in\{2, 3 \ldots, k\}$. 
Polynomial is called Fibonacci-1 if it is a linear combination of Fibonacci monomials. Linear space of Fibonacci-1 polynomials is denoted by $\mathbb{C}^{F}_{1}[x_i]$. There is natural bigradation on this space
\beq
\label{bigradation}
\begin{aligned}
& \deg_z \left(x_{j_1}\,x_{j_2}\ldots x_{j_k}\right) = k,
\\
& \deg_q \left(x_{j_1}\,x_{j_2}\ldots x_{j_k}\right) = -\,j_1 - \,j_2 - \ldots - \,j_k.
\end{aligned}
\ee
\end{definition}
With relation $e^2(z) = 0$ one might prove the following
\begin{lemma}
\ben
W_{0} \myeq \mathbb{C}[e_i \:|\: i\in\mathbb{Z}]\,|0\rangle = \mathbb{C}[e_i \:|\: i \leq - 1]\,|0\rangle = \mathbb{C}^{F}_{1}[e_i \:|\: i \leq  - 1]\,|0\rangle.
\een
\end{lemma}
and by the Weyl group action
\begin{lemma}
\label{lemmaBasicSubspaceSL}
\ben
W_{j} \myeq \mathbb{C}[e_i \:|\: i\in\mathbb{Z}]\,|j\sqrt{2}\rangle = \mathbb{C}[e_i \:|\: i \leq -2j - 1]\,|j\sqrt{2}\rangle = \mathbb{C}^{F}_{1}[e_i \:|\: i \leq -2j - 1]\,|j\sqrt{2}\rangle.
\een
\end{lemma}
Then using character formulas \eqref{BasicSubspaceCharacter} and \eqref{BasicSubspaceCharacter_j} we get
\ben
W_{j} = \mathbb{C}^{F}_{1}[e_i \:|\: i \leq -2j - 1]\,|j\sqrt{2}\rangle \simeq \mathbb{C}^{F}_{1}[e_i \:|\: i \leq -2j - 1].  
\een
Isomorphism of the vector spaces $\mathbb{C}^{F}_{1}[e_i \:|\: i \leq -2j - 1]\simeq \mathbb{C}^{F}_{1}[e_i \:|\: i \leq -2j - 1]\,|j\sqrt{2}\rangle$, $\: g \to g|j\sqrt{2}\rangle$ follows from coincidence of the left and right parts characters. Character of $\mathbb{C}^{F}_{1}[e_i \:|\: i \leq - 1]$ with respect to bigradation \eqref{bigradation} is calculated in \hyperref[Appendix A]{Appendix A}. For general $j$ the statement is obtained by the shift of $e_i$'s indices.  
\\
Therefore, identifying
\beq
\label{IdentificationVacuum}
|j\sqrt{2}\rangle \rightarrow e_{-2j + 1}\,e_{-2j + 3}\,e_{-2j + 5}\ldots,
\ee
we obtain basis in $W = L_{(0 , 1)}$ of semi-infinite monomials
\ben
e_{i_1}\,e_{i_2}\,e_{i_3}\,\ldots,
\een
such that
\begin{enumerate}
    \item $i_1 < i_2 < i_3 < \cdots$ ;
    \item $i_{k + 1} - i_k \geq 2$;
    \item $i_{k} = 1 \mod 2 \quad \text{for } k \gg 1$;
    \item $i_{k + 1} - i_k = 2 \quad \text{for } k \gg 1$.
\end{enumerate}
We prove Lemma \ref{lemmaBasicSubspaceSL} using realization of standard module inside of the zero charge subspace $\mathfrak{F}^{(0)}$. Then character formula \eqref{BasicSubspaceCharacter} is proved in an alternative way. Repeating argument \eqref{CharacterOfSum}, we get desired semi-infinite construction of $L_{(0, 1)}$.

\section{Fermionic Fock space}

\begin{definition}
$V$ is defined as infinite dimensional complex vector space spanned by vectors ${\psi_k,\, k\in\mathbb{Z}}$. $\psi_m^*$ are elements of the bounded dual space defined as:
$$
\psi_m^*(\psi_n) = \delta_{m + n,\, 0}.
$$
\end{definition}
\begin{definition}
\emph{Fermionic Fock space} (semi-infinite wedge sum) $\mathfrak{F}$ is infinite dimensional complex vector space spanned by formal semi-infinite  antisymmetric elements $\psi_{k_0}\wedge\psi_{k_1}\wedge\psi_{k_2}\ldots$ (\emph{elementary vectors}) such that
\begin{enumerate}
\item $k_0 > k_1 > k_2 > \cdots\,$;
\item $k_i = -i + m$\hspace{2mm} for $i \gg 1$ and some $m\in\mathbb{Z}$.
\end{enumerate}
\end{definition}
Number $m$ in this definition is also called \emph{charge}, \emph{energy} of an element is number
\beq
\label{FockEnergy}
\frac{m(m + 1)}{2} + \sum\limits_{i \geq 0} \left(k_i + i - m\right).
\ee
Therefore, $F$ has natural grading
$$
\mathfrak{F} = \bigoplus_{m \in \mathbb{Z}}\mathfrak{F}^{\,m}, \hspace{1cm} \mathfrak{F}^{\,m} = \bigoplus_{j \in \frac{m(m + 1)}{2} + \mathbb{Z}_+} \mathfrak{F}^{\,m}_j.
$$
From formula \eqref{FockEnergy} it's clear that there is natural bijection between elementary vectors of fixed charge $m$ and energy $j$ and partitions of $j$. Consequently, character of $\mathfrak{F}$ is given by formula
\ben
\ch \mathfrak{F} = \sum\limits_{\substack{m\in\mathbb{Z} \\\\ j\in\frac{m(m + 1)}{2} + \mathbb{Z}_{+}}} \left(\dim \mathfrak{F}^{\,m}_j\right) z^m\,q^j  =\sum\limits_{m \in \mathbb{Z}} \frac{z^m\,q^{\frac{m(m + 1)}{2}}}{\varphi(q)}.
\een 
Fermionic Fock space could be interpreted as Dirac's ``electron sea". In more detail one can view \cite{DiFrancesco}, \cite{dirac1981principles}.
\\
Consider complex Clifford algebra $\Cl$ with generators $\psi_i, \psi^*_j$ and relations\footnote{For example, it can be realized as Clifford algebra associated to vector space $\mathbb{C}[t, t^{-1}]\oplus\mathbb{C}[t, t^{-1}]dt$ with inner product given by residue pairing, see \cite{FBS2001}.}
$$
\{\psi_i, \psi_j\} = \{\psi^*_i, \psi^*_j\} = 0,\quad \{\psi_i, \psi^*_j\} = \delta_{i + j,\, 0}.  
$$
Fermionic Fock space is natural vacuum representation of $\Cl$, generated by vector \\$\psi_0\wedge\psi_{-1}\wedge\psi_{-2}\wedge\ldots \equiv |0\rangle$, such that action of $\psi_i$ is given by wedge product, action of $\psi^*_i$ is given by superderivation:
\ben
\psi_i \to \psi_i\,\wedge\:, \hspace{5mm} \psi_i^* \to \frac{\partial}{\partial\psi_{-i}}.
\een

\section{Fermionic realization of $L_{(0, 1)}$}

Fermionic Fock space $\mathfrak{F}$ is a projective representation of the algebra of infinite matrices with finite numbers of nonzero diagonals
\ben
\Bar{a}_{\infty} = \left\{(a_{ij})\:|\: i, j\in\mathbb{Z},\: a_{ij} = 0\quad |i - j| \gg 0\right\},
\een
with natural action $E_{ij} \to\: \normord{\psi_i\psi^*_{-j}}$, where normal ordering is defined as
\ben
\normord{\psi_i\,\psi^*_{-j}} = 
\begin{cases}
-\psi^*_{-j}\,\psi_i, & \mbox{if } i\leq 0,
\\
\psi_i\,\psi^*_{-j}, & \mbox{otherwise}.
\end{cases}
\een
In other words, there is a level one representation $r\colon a_{\infty}\longrightarrow \End(\mathfrak{F})$, where
\ben
a_{\infty} = \Bar{a}_{\infty} \oplus \mathbb{C}c, \quad [a, b] = ab - ba + \alpha(a, b)\,c, \quad [c, \:\cdot\:] = 0,
\een
is central extension of $\Bar{a}_{\infty}$ with two-cocycle
\ben
\begin{cases}
\alpha(E_{ij}, E_{ji}) = -\alpha(E_{ji}, E_{ij}) = 1 & \mbox{if } i\leq 0,\: j\geq 1,
\\
\alpha(E_{ij}, E_{mn}) = 0 & \mbox{otherwise}.
\end{cases}
\een
$a_{\infty}$ contains Heisenberg algebra $\mathcal{A} = \langle \Lambda_j, c\rangle_{j\in\mathbb{Z}}$ with 
\ben
\Lambda_j = \sum\limits_{k \in \mathbb{Z}}\,E_{k,\, k + j}\,, \quad [\Lambda_i, \Lambda_j] = i\delta_{i + j,\, 0}\,c.
\een
Charge subspaces $\mathfrak{F}^{(m)}$ are irreducible representations over this Heisenberg algebra (and consequently~$a_{\infty}$).
\\
$\widehat{\mathfrak{gl}_n}'$ $\left(\text{particularly } \widehat{\mathfrak{sl}}_2'\right)$ is embedded into $a_{\infty}$ as Lie subalgebra by\footnote{this realisation is described in detail in \cite{KacRaina}.}
\beq
\label{EmbeddingOfsl2}
e_{ij}(k) \to \sum\limits_{s\in\mathbb{Z}} E_{ns + i,\, n(s + k) + j}\,, \quad K\to c
	\ee
and contains $\mathcal{A}$, meaning $\mathfrak{F}^{(m)}$ being irreducible over $\widehat{\mathfrak{gl}}_n'$.
\\
Thus, representation $\tilde{r}\colon\widehat{\mathfrak{sl}}_2'\to\End\left(\mathfrak{F}^{(0)}\right)$ is constructed. It's easy to see that $\widehat{\mathfrak{sl}}_2'$ contains $\Lambda_j$ for odd $j$:
\ben
\Lambda_{2k + 1} = \sum\limits_{j\in\mathbb{Z}} E_{j, j + 2k + 1} = \sum\limits_{l\in\mathbb{Z}} E_{2l, 2l + 2k + 1} + \sum\limits_{l\in\mathbb{Z}} E_{2l + 1, 2l + 1 + 2k + 1} = e_k + f_{k + 1},
\een
and action of $\widehat{\mathfrak{sl}}_2'$ commutes with action of $\Lambda_j$'s for even $j$ (from \eqref{EmbeddingOfsl2} such $\Lambda_j$'s coincide with $e_{11}\left(\frac{j}{2}\right) + e_{22}\left(\frac{j}{2}\right)$ in $a_{\infty}$). Through BF correspondence then\footnote{Realization of $\widehat{\mathfrak{sl}}_2$ through differential operators on $\mathbb{C}[x_1, x_2, x_3, \ldots]$ was firstly obtained in \cite{LepowskyWilson1978}.} 
\ben
\sigma(\mathfrak{F}^{(0)}) = \mathbb{C}[x_1, x_2, x_3, \ldots\,] = \mathbb{C}[x_1, x_3, x_5, \ldots\,]\otimes \mathbb{C}[x_2, x_4, x_6, \ldots\,],  
\een
where $\mathbb{C}[x_1, x_3, x_5, \ldots\,]$ is irreducible $\widehat{\mathfrak{sl}_2}$-module of highest weight $(0, 1)$ and $\mathbb{C}[x_2, x_4, x_6, \ldots\,]$ is the multiplicity space. Then $L_{(0, 1)}$ might be extracted as
\ben
L_{(0, 1)} = \left\{v\in\mathfrak{F}^{0}\:|\:\Lambda_{2k}\,v = 0 \text{ for any } k > 0 \right\}.
\een

\section{Structure of the basic subspace}

Standard basis in $\mathfrak{sl}_2$ is 
\ben
e = 
\begin{pmatrix}
0 & 1 \\
0 & 0
\end{pmatrix},
\quad
f = 
\begin{pmatrix}
0 & 0 \\
1 & 0
\end{pmatrix},
\quad
h = 
\begin{pmatrix}
1 & 0 \\
0 & -1
\end{pmatrix}.
\een
From \eqref{EmbeddingOfsl2} it follows that
\ben
\begin{aligned}
&e_k\equiv e_{12}(k) \to \sum\limits_{s\in\mathbb{Z}}\,E_{2s + 1,\, 2(s + k) + 2}\,, \quad f_k\equiv e_{21}(k) \to \sum\limits_{s\in\mathbb{Z}}\,E_{2s + 2,\, 2(s + k) + 1}\,,
\\
&h_k\equiv e_{11}(k) - e_{22}(k) \rightarrow \sum\limits_{s\in\mathbb{Z}}\,\left(E_{2s + 1,\, 2(s + k) + 1} - E_{2s + 2,\, 2(s + k) + 2}\right).
\end{aligned}
\een
\begin{remark}
Directly form action $E_{ij} \to :\psi_i\psi^*_{-j}:$ it's clear that operator $e_k$ consequently shifts even indices of $\psi$'s by $-2k - 1$, for example
\ben
e_{0}\,|0\rangle = e_{0}\,\psi_0\wedge\psi_{-1}\wedge\psi_{-2}\wedge\psi_{-3}\wedge\ldots = 0,  
\een
\ben
e_{-1}\,|0\rangle = e_{-1}\,\psi_0\wedge\psi_{-1}\wedge\psi_{-2}\wedge\psi_{-3}\wedge\ldots = \psi_1\wedge\psi_{-1}\wedge\psi_{-2}\wedge\psi_{-3}\wedge\ldots,  
\een
\ben
\begin{aligned}
&e_{-3}\,|0\rangle = e_{-3}\,\psi_0\wedge\psi_{-1}\wedge\psi_{-2}\wedge\psi_{-3}\wedge\psi_{-4}\wedge\psi_{-5}\wedge\ldots 
\\
& = \psi_5\wedge\psi_{-1}\wedge\psi_{-2}\wedge\psi_{-3}\wedge\psi_{-4}\wedge\psi_{-5}\wedge\ldots + \psi_0\wedge\psi_{-1}\wedge\psi_{3}\wedge\psi_{-3}\wedge\psi_{-4}\wedge\psi_{-5}\wedge\ldots
\\
& + 
\psi_0\wedge\psi_{-1}\wedge\psi_{-2}\wedge\psi_{-3}\wedge\psi_{1}\wedge\psi_{-5}\wedge\ldots
\\
& = \psi_5\wedge\psi_{-1}\wedge\psi_{-2}\wedge\psi_{-3}\wedge\psi_{-4}\wedge\psi_{-5}\wedge\ldots + \psi_{3}\wedge\psi_0\wedge\psi_{-1}\wedge\psi_{-3}\wedge\psi_{-4}\wedge\psi_{-5}\wedge\ldots
\\
& + 
\psi_{1}\wedge\psi_0\wedge\psi_{-1}\wedge\psi_{-2}\wedge\psi_{-3}\wedge\psi_{-5}\wedge\ldots\:.
\end{aligned}
\een
Then up to a constant
\ben
\begin{aligned}
&|n\sqrt{2}\rangle \sim \psi_{2n + 1}\wedge\psi_{2n - 1}\wedge\psi_{2n - 3}\wedge\ldots\wedge\psi_{- 2n - 1}\wedge\psi_{- 2n - 2}\wedge\psi_{- 2n - 3}\wedge\ldots \hspace{5mm} n\in\mathbb{N},
\\
&|-n\sqrt{2}\rangle \sim \psi_{2n}\wedge\psi_{2n - 2}\wedge\psi_{2n - 4}\wedge\ldots\wedge\psi_{- 2n}\wedge\psi_{- 2n - 2}\wedge\psi_{- 2n - 3}\wedge\psi_{- 2n - 4}\wedge\ldots \hspace{5mm} n\in\mathbb{N},
\\
&|0\rangle \sim \psi_0\wedge\psi_{-1}\wedge\psi_{-2}\wedge\ldots\:.
\end{aligned}
\een
\end{remark}
\begin{lemma}
Any nontrivial Fibonacci monomial $g\in \mathbb{C}^{F}_{1}[e_i \:|\: i \leq -1]$ acts nontrivially on $|0\rangle$.
\end{lemma}
\begin{proof}
Let $e_{-i_k}\,e_{-i_{k - 1}}\ldots\,e_{-i_1} \in \mathbb{C}^{F}_{1}[e_i \:|\: i \leq -1]$ with $i_1 < i_2 < \cdots < i_k$. 
\beq
\label{sum_of_el_v}
e_{-i_k}\,e_{-i_{k - 1}}\ldots\,e_{-i_1}|0\rangle
\ee
is the sum of elementary vectors, which is nonzero as soon as there is elementary vector
\beq
\label{Qel_vector}
\begin{aligned}
&Q(e_{-i_1}\,e_{-i_2}\ldots\,e_{-i_k}) \equiv
\\
&\psi_{2i_1 - 1}\wedge\psi_{- 1}\wedge\psi_{- 2 + 2i_2 - 1}\wedge\psi_{-3}\wedge\psi_{- 4 + 2i_3 - 1}\wedge
\ldots\wedge\psi_{-2k + 3}\wedge\psi_{-2(k - 1) + 2i_k - 1}\wedge\psi_{-2k + 1}\wedge\psi_{-2k}\ldots\:,
\end{aligned}
\ee
which doesn't contract with any other elementary vector. Indeed, the only way to get this vector in \eqref{sum_of_el_v} is to act by $e_{i_1}$ on $\psi_0$, $e_{i_2}$ on $\psi_{-2}$, \ldots, $e_{i_k}$ on $\psi_{-2k + 2}$.  
\end{proof}
\begin{remark}
    Elementary vector \eqref{Qel_vector} is nontrivial as soon as there is a Fibonacci condition ${i_j - i_{j - 1} > 1}$, consequently
    \ben
-1 < 2i_1 - 1 < 2i_2 - 2 - 1 < \cdots.
    \een
\end{remark}
\begin{lemma}
Any nontrivial Fibonacci-1 polynomial $g\in \mathbb{C}^{F}_{1}[e_i \:|\: i \leq -1]$ acts nontrivially on $|0\rangle$.
\end{lemma}
\begin{proof}
It's enough to prove this fact for homogeneous polynomial with $\deg_z g = n$, $\deg_q g =  m$. $g$ is finite linear combination of Fibonacci monomials. There is natural identification of Fibonacci monomial $e_{-i_n}\,e_{-i_{n - 1}}\ldots\,e_{-i_1} \in \mathbb{C}^{F}_{1}[e_i \;|\; i \leq -1]$ and partition $(i_n, i_{n - 1}, \ldots, i_{1})$. Reflected lexicographic order on these partitions gives total order on Fibonacci monomials $\tilde{g}$ with $\deg_z \tilde{g} = n$, $\deg_q \tilde{g} =  m$. Let $e_{-i_n}\,e_{-i_{n - 1}}\ldots\,e_{-i_1}$ be the maximal Fibonacci monomial among nontrivial monomials in $g$ in sense of this order. Then $Q(e_{-i_n}\,e_{-i_{n - 1}}\ldots\,e_{-i_1})$ is an elementary vector which appears in expansion of $g|0\rangle$ and doesn't contract with other elementary vectors.
\end{proof}
Then we know $W_{0}^{\sqrt{2}} \simeq \mathbb{C}^{F}_{1}[e_i \:|\: i \leq -1]$ and by the same argument
$$
W_{j}^{\sqrt{2}}~\simeq~\mathbb{C}^{F}_{1}[e_i \:|\: i \leq -2j - 1].
$$
Using expressions \eqref{CharacterOfSum} and \eqref{IdentificationVacuum} we get desired semi-infinite construction of $L_{(0, 1)}$.
\\
By the same procedure, realizing $L_{(1, 1)}$ inside of $\mathfrak{F}^{(1)}$ one can prove that $L_{(1, 1)}$ has the basis of semi-infinite monomials
\ben
e_{i_1}\,e_{i_2}\,e_{i_3}\,\ldots,
\een
such that
\begin{enumerate}
    \item $i_1 < i_2 < i_3 < \cdots$,
    \item $i_{k + 1} - i_k \geq 2$,
    \item $i_{k} = 0 \mod 2 \quad \text{for } k \gg 1$;
    \item $i_{k + 1} - i_k = 2 \quad \text{for } k \gg 1$.
\end{enumerate}

\clearpage
\bibliography{bibliography}{}

\begin{thebibliography}{1}

\bibitem{DiFrancesco}
{\sc Di~Francesco, P., Mathieu, P., and Senechal, D.}
\newblock {\em {Conformal Field Theory}}.
\newblock Graduate Texts in Contemporary Physics. Springer-Verlag, New York, 1997.

\bibitem{dirac1981principles}
{\sc Dirac, P.}
\newblock {\em The Principles of Quantum Mechanics}.
\newblock Comparative Pathobiology - Studies in the Postmodern Theory of Education. Clarendon Press, 1981.

\bibitem{FBS2001}
{\sc Frenkel, E., Ben-Zvi, D., and Society, A.~M.}
\newblock {\em Vertex algebras and algebraic curves}.
\newblock American Mathematical Society, 2001.

\bibitem{KacRaina}
{\sc Kac, V.~G.}
\newblock {\em {Bombay lectures on highest weight representations of infinite dimensional Lie algebras / by V.G. Kac, A.K. Raina.}}
\newblock Advanced series in mathematical physics ; vol. 2. World Scientific, 1987.

\bibitem{LepowskyWilson1978}
{\sc Lepowsky, J., and Wilson, R.~L.}
\newblock {Construction of the Affine Lie Algebra A1(1)}.
\newblock {\em Commun. Math. Phys. 62\/} (1978), 43--53.

\bibitem{FS}
{\sc Stoyanovsky, A.~V., and Feigin, B.~L.}
\newblock {Functional models for representations of current algebras and semi-infinite Schubert cells}.
\newblock {\em Functional Analysis and Its Applications 28}, 1 (1994), 55--72.

\end{thebibliography}

\clearpage
\begin{appendices}

 \renewcommand\thetable{\thesection\arabic{table}}
  \renewcommand\thefigure{\thesection\arabic{figure}}
 	
\section{Character formula for space of Fibonacci-1 polynomials}
\label{Appendix A}
Character of $\mathbb{C}^F_1[e_i \;|\; i < 0]$ with respect to bigradation might be derived from the character of $\mathbb{C}^F_1[e_{-1}, e_{-2}, \ldots, e_{-(N - 1)}]$ setting $N\to +\infty$. From Definition \ref{Fibb_monomials} it's clear that character of $\mathbb{C}^F_1[e_{-1}, e_{-2}, \ldots, e_{-(N - 1)}]$ is
\beq
\label{qBinom2}
\sum\limits_{n = 0}^{\infty}\,\sum\limits_{m = 0}^{\infty}\,p(n \,|\, m \text{ distinct parts, each } \leq N - 1, \text{ adjacent parts differs } \geq 2)\,z^m\,q^n.
\ee
Coefficient on $z^m$ term in this sum equals
\ben
q^{1 + 3 + \ldots + (2m - 1)}\,\qbinom{N - m}{m} = q^{m^2}\, \qbinom{N - m}{m},
\een
what is clear from ``cutting" of the Young diagram, illustrated by Figure \ref{Dissection_Diagram_2}.
Thus,
\ben
\ch \mathbb{C}^F_1[e_{-1}, e_{-2}, \ldots, e_{-(N - 1)}] = \sum\limits_{m = 0}^{[\frac{N}{2}]}\,z^m\,q^{m^2}\,\qbinom{N - m}{m}_q.
\een
Setting $N\to +\infty$, we obtain
\ben
\ch \mathbb{C}^F_1[e_i \:|\: i < 0] = \sum\limits_{m = 0}^{\infty}\, \frac{z^m\,q^m}{(q)_{m}},
\een
which coincides with \eqref{BasicSubspaceCharacter}.

\begin{figure}[h!]
\begin{center}
\includegraphics[scale = 0.7]{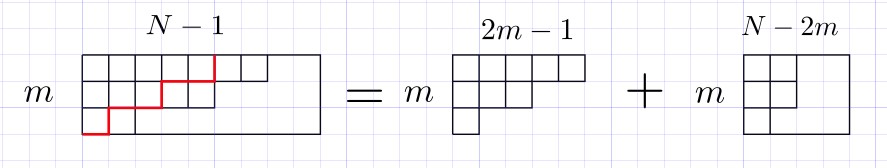}
\caption{Illustration of \eqref{qBinom2}: any Young diagram corresponding to partition into ${m \: \text{distinct}\: \leq N - 1}$ parts with adjacent differing $\geq 2$ might be obtained by ascribing Young diagram contained in ${m\times (N - 2m)}$ rectangle to $(1, 3, \ldots, 2m - 1)$ shape  from the right.}
\label{Dissection_Diagram_2}
\end{center}
\end{figure}

\end{appendices}

\end{document}